\documentclass[12pt]{amsart}

\usepackage{graphicx}
\usepackage{eucal}
\usepackage{amsfonts,amssymb,amscd}

\newtheorem{theorem}{Theorem}[section]
\newtheorem{lemma}[theorem]{Lemma}
\theoremstyle{definition}

\newtheorem{proposition}[theorem]{Proposition}
\newtheorem{corollary}[theorem]{Corollary}

\theoremstyle{remark}
\newtheorem{remark}[theorem]{Remark}
\numberwithin{equation}{section}
\usepackage{graphicx}
\usepackage{hyperref}
\usepackage{geometry}
\usepackage[all]{xy}
\usepackage{algorithm2e, amsmath, amssymb, cite, hyperref, xcolor}

\begin{document}
\title{A Gradient Descent Method for The Dubins Traveling Salesman Problem}
\author{D. Kirszenblat \and
	J. Ayala \and
	J. H. Rubinstein
}
              \address{School of Mathematics and Statistics, The University of Melbourne, Melbourne, Victoria 3010, Australia.}
              \email{d.kirszenblat@student.unimelb.edu.au}           
              \address{Universidad de Tarapac\'a, Instituto de Alta Investigaci\'on, Casilla 7D, Arica, Chile.}
              \address{School of Mathematics and Statistics, The University of Melbourne, Melbourne, Victoria 3010, Australia.}
              \email{jayalhoff@gmail.com}
              \address{School of Mathematics and Statistics, The University of Melbourne, Melbourne, Victoria 3010, Australia.}
              \email{joachim@unimelb.edu.au}
\date{}
\baselineskip=20 true pt
\maketitle \baselineskip=1\normalbaselineskip
\begin{abstract}
We propose a combination of a bounding procedure and gradient descent method for solving the Dubins traveling salesman problem, that is, the problem of finding a shortest curvature-constrained tour through a finite number of points in the euclidean plane. The problem finds applications in path planning for robotic vehicles and unmanned aerial vehicles, where a minimum turning radius prevents the vehicle from taking sharp turns. In this paper, we focus on the case where any two points are separated by at least four times the minimum turning radius, which is most interesting from a practical standpoint. The bounding procedure efficiently determines the optimal order in which to visit the points. The gradient descent method, which is inspired by a mechanical model, determines the optimal trajectories of the tour through the points in a given order, and its computation time scales linearly with the number of points. In experiments on nine points, the bounding procedure typically explores no more than a few sequences before finding the optimal sequence, and the gradient descent method typically converges to within 1\% of optimal in a single iteration.
\end{abstract}

\section{Introduction}
\label{intro}
The Dubins traveling salesman problem (DSTP) is to construct a shortest curvature-constrained tour through $n$ distinct points ${\bf x}_1, {\bf x}_2, \ldots, {\bf x}_n$ in the euclidean plane $\mathbb{R}^2$, where the optimal sequencing of the points is to be determined along with the trajectories of the tour through the points. The problem finds applications in path planning for robotic vehicles and unmanned aerial vehicles, where a minimum turning circle prevents the vehicle from taking sharp turns \cite{Berglund, Cohen, Otto, Tang}. In this paper, we will focus on the long path case, that is, the case where any two directed points are separated by at least four times the minimum turning radius $\rho > 0$. The long path case is both easier to address from a computational standpoint and most interesting from a practical standpoint.

To be clear, for $i = 1, 2, \ldots, n$, the tour is tangent to the unit vector $(\cos \theta_i, \sin \theta_i)$ based at the point ${\bf x}_i$, where the angle $\theta_i \in S^1$ is a variable to be determined. Taken together, the point ${\bf x}_i$ and angle $\theta_i$ are considered as a directed point. Due to a result of Dubins, it is known that any shortest curvature-constrained curve connecting two directed points $({\bf x}_1, \theta_1)$ and $({\bf x}_2, \theta_2)$ is a subtype of one of two types, namely, $CCC$ and $CSC$, where the symbol $C$ designates a circular arc of maximal curvature $\rho^{-1}$ and the symbol $S$ designates a straight segment \cite{Dubins}. In the long path case, a shortest curvature-constrained curve is of type $CSC$ for a generic choice of two directed points. 

The problem has a combinatorial aspect which concerns the order in which the tour visits the $n$ points in the euclidean plane. The problem also has a continuous aspect which, for a given sequence of points, concerns the directions of the unit tangent vectors. We note that in the limit, as the minimum turning radius $\rho$ tends toward zero, we recover the euclidean traveling salesman problem (ETSP). A common approach in the literature is to use the solution to the ETSP as a basis for constructing a curvature-constrained solution to the DTSP \cite{Macharet, Ny}. The subproblem of computing a shortest curvature-constrained path through a given sequence of points has been addressed using, for example, convex optimisation and optimal control theory \cite{Goaoc, Kaya, Lee, Ma, Manyam, Rathinam}.

Our contribution to solving the Dubins traveling salesman problem is twofold.
\begin{enumerate}
\item We propose a bounding procedure for efficiently exploring the space of sequences of points in order to identify an optimal sequence.
\item We propose a practical gradient descent method for determining the trajectories of the tour through a given sequence of points, where the computation time scales linearly with the number of points.
\end{enumerate}
The gradient descent method is inspired by a mechanical model. Analogous models have been used with success in the study of curvature-constrained networks and ribbon knot diagrams \cite{Kirszenblat2, Ayala}. The efficacy of our algorithm is supported by experimental results.

The paper is structured as follows. In Section \ref{sec:2}, we derive an expression for the gradient of the length of a tour through a given sequence of points. In Section \ref{sec:3}, we propose a seed solution for the gradient descent method. In Section \ref{sec:4}, we bound the optimality gap between the objective value of the incumbent solution and that of a locally optimal solution of the gradient descent method. In Section \ref{sec:5}, we bring together the results of the previous sections in the gradient descent method. In Section \ref{sec:6}, we propose a bounding procedure which uses the gradient descent method as a subroutine. In Section \ref{sec:7}, we present some experimental results. Finally, in Section \ref{sec:8}, we explain how to generalise the approach when the positions of the points may be varied.

\section{Computing the gradient of the length function}
\label{sec:2}
We open this section on computing the gradient of the length of a tour through a given sequence of points by presenting a mechanical model, which has been considered but not fully exploited by Goaoc et al. Where our results overlap with theirs, we provide alternate derivations, which are comparatively direct and intuitive in light of the mechanical model. To that end, consider an elastic band which runs through $n$ pairs of disks, as is illustrated in Figure \ref{fig:1} (left). The elastic band represents the tour for a fixed sequence of nodes, where each pair of disks is centred at a point representing a node of the tour. Each pair of disks serves to ensure that the elastic band satisfies the curvature constraint as it runs through the point between the two disks. Moreover, each pair of disks is free to rotate in the plane about the point between the two disks, as the direction of the tour through the corresponding node is a variable to be determined. Under the force of the tension in the elastic band, the elastic band settles into an equilibrium configuration, which minimises the potential energy of the system and, by implication, the length of the elastic band. Two properties of the mechanical model should be clear. First, each pair of disks pivots in response to a torque induced by the elastic band. Second, when the length of the elastic band is minimised, the two torques exerted upon a pair of disks -- one clockwise and the other counterclockwise -- are in balance.

\begin{figure}[h]
    \centering
    \includegraphics[width=\textwidth]{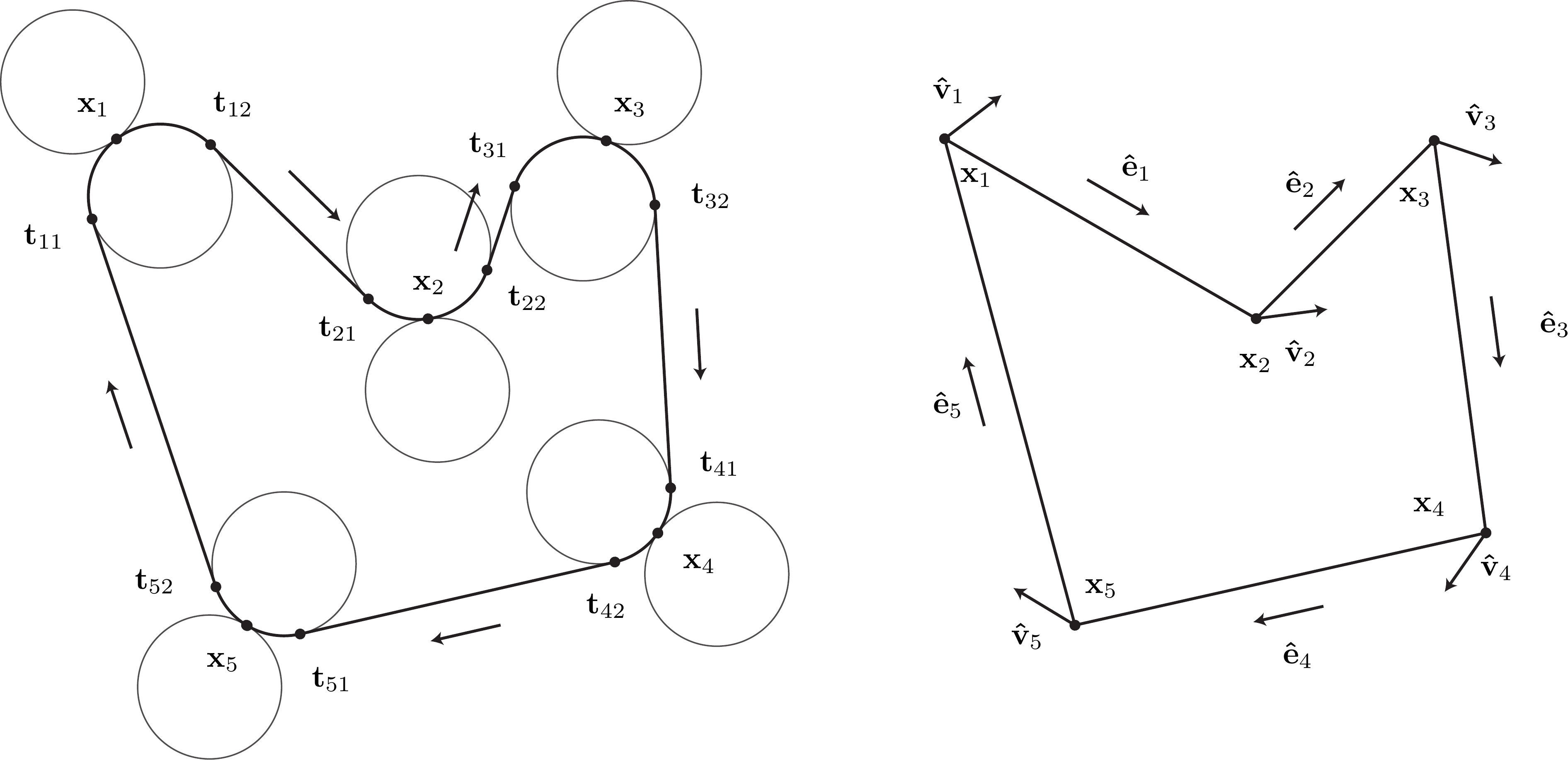}
    \caption{Left: A curvature-constrained tour may be modelled by a mechanical device comprising an elastic band which runs through a number of pairs of freely rotating disks. Right: In the limit, as the minimum turning radius tends toward zero, we recover the euclidean traveling salesman problem.}
    \label{fig:1}
\end{figure}

In order to derive an expression for the gradient, we borrow the following lemma with adapted notation from the study of curvature-constrained networks \cite[Proposition 1]{Kirszenblat2}. For ease of notation and without loss of generality, we scale the problem, so that $\rho = 1$.

\begin{lemma}
\label{lem:1}
Let ${\bf p}_1 {\bf p}_2 {\bf p}_3$ denote a Dubins curve of type $CS$, and let ${\bf \hat u}$ denote the normalisation of the vector ${\bf p}_3 {\bf p}_2$. If the point ${\bf p}_3$ is varied with velocity ${\bf v}$, then the first variation of the length of the curve is given by

\begin{equation*}
-{\bf v} \cdot {\bf \hat u}.
\end{equation*}
\end{lemma}

\begin{proof}
Refer to Figure \ref{fig:2} (left). Let $\phi$ denote the length of the arc ${\bf p}_1 {\bf p}_2$. When the point ${\bf p}_3$ is varied, the first variation of the length of the arc ${\bf p}_1 {\bf p}_2$ is given by $\phi'$. On the other hand, the first variation of the length of the straight segment ${\bf p}_3 {\bf p}_2$ is given by ${\bf u}' \cdot {\bf \hat u}$. Summing these two terms, we obtain for the first variation of the length of the curve ${\bf p}_1 {\bf p}_2 {\bf p}_3$ the expression

\begin{align*}
\phi' + {\bf u}' \cdot {\bf \hat u} &= \phi' + ({\bf p}_2 - {\bf p}_3)' \cdot {\bf \hat u}\\
&= \phi' + ({\bf p}_2' - {\bf p}_3') \cdot {\bf \hat u}\\
&= \phi' + {\bf p}_2' \cdot {\bf \hat u} - {\bf p}_3' \cdot {\bf \hat u}\\
&= \phi' - \phi' - {\bf v} \cdot {\bf \hat u}\\
&= -{\bf v} \cdot {\bf \hat u}.
\end{align*}
\end{proof}

\begin{figure}[h]
    \centering
    \includegraphics[width=\textwidth]{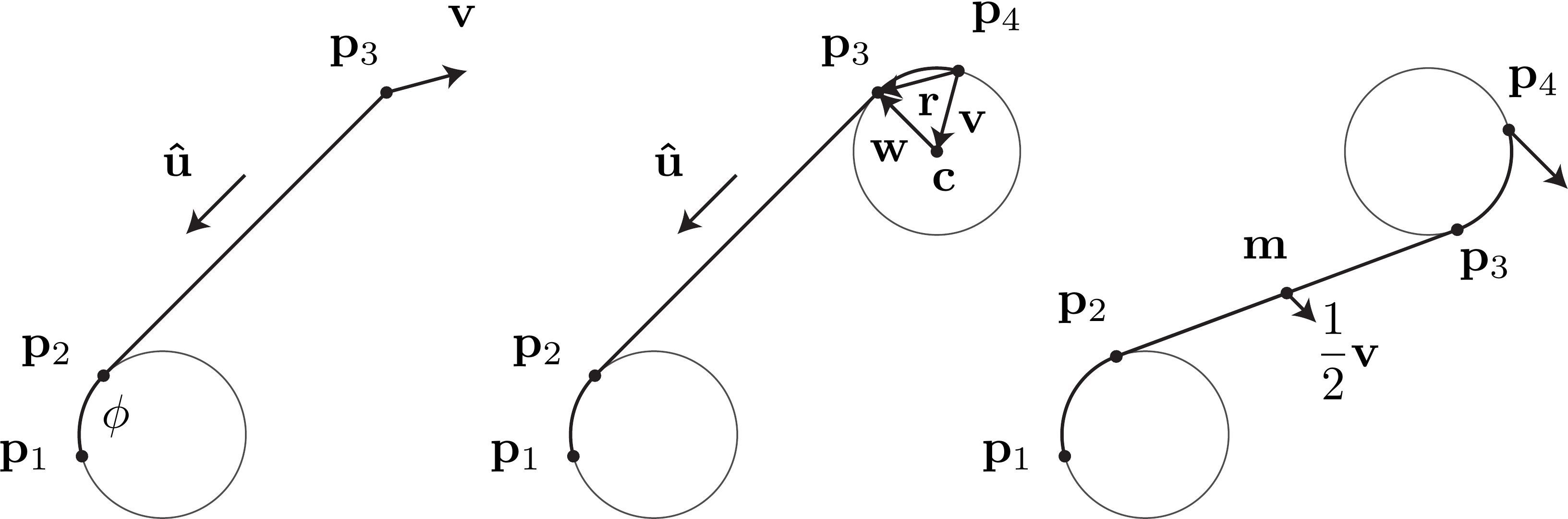}
    \caption{Left: The point ${\bf p}_3$ is perturbed in the direction of a vector ${\bf v}$. Middle: The disk with centre ${\bf c}$ is rotated about the point ${\bf p}_4$. Right: The point ${\bf p}_4$ is perturbed in the direction of a vector ${\bf v}$.}
    \label{fig:2}
\end{figure}

\begin{remark}
The unit vector ${\bf \hat u}$ may be interpreted as the force exerted upon the point ${\bf p}_3$ by the elastic representing the Dubins curve.
\end{remark}

The following theorem was first presented in the unpublished honours thesis of the first author \cite[Proposition 5.2]{Kirszenblat1}.

\begin{theorem}
\label{the:1}
Let ${\bf p}_1 {\bf p}_2 {\bf p}_3 {\bf p}_4$ denote a Dubins curve of type $CSC$. Let ${\bf r}$ denote the vector ${\bf p}_4 {\bf p}_3$ and ${\bf \hat u}$ the normalisation of the vector ${\bf p}_3 {\bf p}_2$. Finally, let ${\bf \hat z}$ denote the unit vector which points out of the plane of the page. If the angle associated with the point ${\bf p}_4$ is varied with angular velocity $\omega$, then the first variation of the length of the curve is given by

\begin{equation*}
-\omega({\bf r} \times {\bf \hat u}) \cdot {\bf \hat z}.
\end{equation*}
\end{theorem}

\begin{proof}
Refer to Figure \ref{fig:2} (middle). The arc ${\bf p}_3 {\bf p}_4$ belongs to a circle with centre ${\bf c}$. Let ${\bf v}$ and ${\bf w}$ denote the vectors ${\bf p}_4 {\bf c}$ and ${\bf c} {\bf p}_3$, respectively. Note that a rotation of the circle with centre ${\bf c}$ about the point ${\bf p}_4$ by an angle of $t \omega$ is the same as the composition of a rotation about the point ${\bf c}$ by an angle of $t \omega$ and a translation by ${\bf v}_t - {\bf v}_0$, where ${\bf v}_0$ is the vector ${\bf v}$ at time 0 and ${\bf v}_t$ is the vector ${\bf v}$ at time $t$. By Lemma \ref{lem:1}, the first variation of the length of the curve ${\bf p}_1 {\bf p}_2 {\bf p}_3$ is given by

\begin{align*}
-({\bf v}_t - {\bf v}_0)' \cdot {\bf \hat u} &= -{\bf v}_t' \cdot {\bf \hat u}\\
&= -\omega ({\bf \hat z} \times {\bf v}) \cdot {\bf \hat u},
\end{align*}

where we have dropped the dependence on $t$. On the other hand, the first variation of the length of the arc ${\bf p}_3 {\bf p}_4$ is given by

\begin{equation*}
-\omega = -\omega ({\bf \hat z} \times {\bf w}) \cdot {\bf \hat u}.
\end{equation*}

Summing these two terms, we obtain for the first variation of the length of the curve ${\bf p}_1 {\bf p}_2 {\bf p}_3 {\bf p}_4$ the expression

\begin{align*}
-\omega ({\bf \hat z} \times {\bf v}) \cdot {\bf \hat u} - \omega ({\bf \hat z} \times {\bf w}) \cdot {\bf \hat u} &= -\omega ({\bf \hat z} \times ({\bf v} + {\bf w})) \cdot {\bf \hat u}\\
&= -\omega ({\bf \hat z} \times {\bf r}) \cdot {\bf \hat u}\\
&= -\omega ({\bf r} \times {\bf \hat u}) \cdot {\bf \hat z}.
\end{align*}

\end{proof}

\begin{remark}
The expression ${\bf r} \times {\bf \hat u}$ may be interpreted as the torque exerted upon the disk with centre ${\bf c}$ by the elastic representing the Dubins curve.
\end{remark}

We can now provide an expression for the partial derivatives of the length of a tour through a given sequence of points as the angles $\theta_1, \theta_2, \ldots, \theta_n$ are varied. Consider the map $L: (S^1)^n \to \mathbb{R}^+$ which assigns to an $n$-tuple of angles the length of the corresponding curvature-constrained tour through a given sequence of points ${\bf x}_1, {\bf x}_2, \ldots, {\bf x}_n$.

\begin{theorem}
\label{the:2}
The $i$th partial derivative of the length of the tour may be expressed as

\begin{equation*}
\frac{\partial L}{\partial \theta_i} = -\tau_i,
\end{equation*}

where $\tau_i$ may be interpreted as the signed magnitude of the total torque exerted upon the $i$th pair of disks by the elastic band representing the Dubins curve.
\end{theorem}

\begin{proof} Referring to Figure \ref{fig:1}, for $i = 1, 2, \ldots, n$, let ${\bf \hat u}_i$ denote the unit vector which points along the $i$th straight segment and has the same orientation as the tour. Denote its initial and final points by ${\bf t}_{i2}$ and ${\bf t}_{\mbox{mod}(i + 1, n)1}$, respectively. Define the vectors ${\bf r}_{i1}$ and ${\bf r}_{i2}$ as follows:

\begin{align*}
{\bf r}_{i1} &= {\bf t}_{i1} - {\bf x}_i,\\
{\bf r}_{i2} &= {\bf t}_{i2} - {\bf x}_i.
\end{align*}

Following Theorem \ref{the:1}, the negative of the first variation of the length of the tour as the $i$th angle is varied is given by

\begin{equation*}
\tau_i = ({\bf r}_{i2} \times {\bf \hat u}_i - {\bf r}_{i1} \times {\bf \hat u}_{\mbox{mod}(i, n, 1)}) \cdot {\bf \hat z},
\end{equation*}

where $\mbox{mod}(i, n, 1) = i - n \lfloor\frac{i - 1}{n}\rfloor$.
\end{proof}

\begin{corollary}
When the length of the tour through a fixed sequence of nodes is minimised, for $i = 1, 2, \ldots, n$, we have

\begin{equation*}
{\bf r}_{i2} \times {\bf \hat u}_i = {\bf r}_{i1} \times {\bf \hat u}_{\mbox{mod}(i, n, 1)}.
\end{equation*}

In other words, the torques exerted by the elastic band on the $i$th disk are in balance.
\end{corollary}

Writing ${\boldsymbol \theta} = (\theta_1, \theta_2, \ldots, \theta_n)$ and ${\boldsymbol \tau} = (\tau_1, \tau_2, \ldots, \tau_n)$, the gradient may be expressed succinctly as

\begin{equation*}
\nabla L({\boldsymbol \theta}) = -{\boldsymbol \tau}.
\end{equation*}

\section{The seed solution}
\label{sec:3}
The limiting case of the problem -- as the minimum turning radius $\rho$ is taken to zero -- suggests a seed solution which performs very well in experiments. Refer to Figure \ref{fig:1} (right) for an illustration. Let ${\bf \hat e}_i$ denote the unit vector pointing from node $i$ to node $\mbox{mod}(i + 1, n)$. Moreover, let ${\bf v}_i = {\bf \hat e}_{\mbox{mod}(i - 1, n, 1)} + {\bf \hat e}_i$ and ${\bf \hat v}_i = {\bf v}_i/\|{\bf v}_i\|$. As explained in the previous section, when the length of the tour through a fixed sequence of points is minimised, the two torques acting on a pair of disks are of equal magnitude and opposite sign. This equilibrium condition translates into the requirement that the two arcs adjacent to a node have the same length and orientation. In the limiting case -- as the minimum turning radius $\rho$ tends toward zero -- this means that the unit tangent vector $(\cos \theta_i, \sin \theta_i)$ tends toward ${\bf \hat v}_i$. The vectors ${\bf v}_1, {\bf v}_2, \ldots, {\bf v}_n$ determine a choice of angles $(\theta_1, \theta_2, \ldots, \theta_n)$ which may be used to seed the gradient descent method. As we will see in Section \ref{sec:7}, it turns out that this seed solution performs very well in practice.

\section{Bounding the optimality gap}
\label{sec:4}
In this section, we use the local convexity of the length function to bound the optimality gap between the objective value of the incumbent solution and that of the optimal solution obtained via gradient descent -- at this stage, we continue to focus on the length of a tour through a fixed sequence of points. For this purpose, we provide alternate and more concise proofs of the following proposition and theorem from the work of Goaoc et al. (cf. \cite[Propositions 35 and 36, pp. 30--41]{Goaoc}).

\begin{proposition}
\label{pro:1}
If both circular arcs of a $CSC$ curve connecting two directed points are strictly shorter than $\pi$, then it is the shortest amongst all such $CSC$ curves.
\end{proposition}

\begin{proof}
Denote the $CSC$ curve ${\bf p}_1 {\bf p}_2 {\bf p}_3 {\bf p}_4$ by $\gamma_1$ and suppose that the arcs ${\bf p}_1 {\bf p}_2$ and ${\bf p}_3 {\bf p}_4$ are strictly shorter than $\pi$. Let $\gamma_2$ denote a distinct $CSC$ curve connecting the directed points $({\bf p}_1, \theta_1)$ and $({\bf p}_4, \theta_4)$. There are two cases to consider. In case 1, the curves $\gamma_1$ and $\gamma_2$ each contain a circular arc belonging to the same circle and, in particular, the arc of $\gamma_1$ is longer than that of $\gamma_2$. See Figure \ref{fig:3} for an illustration. Case 2 covers all other possibilities.

Case 2: Since the argument concerning case 2 is simpler than that concerning case 1, we will consider it first. If the circular arcs of $\gamma_1$ and $\gamma_2$ departing from the point ${\bf p}_1$ belong to different circles, then construct a line $l_1$ between these two circles. See Figure \ref{fig:4}. Do the same for the point ${\bf p}_4$. If we have constructed a line through the point ${\bf p}_1$, then we reflect the first part of the curve $\gamma_2$ through this line. Similarly, if we have constructed a line through the point ${\bf p}_4$, then we reflect the last part of the curve $\gamma_2$ through this line. Thus, we produce a curve $\tilde{\gamma}_2$, whose length is equal to that of $\gamma_2$. Observe that the curves $\gamma_1$ and $\tilde{\gamma}_2$ share the circular arcs ${\bf p}_1 {\bf p}_2$ and ${\bf p}_3 {\bf p}_4$. However, the part of $\gamma_1$ between the points ${\bf p}_2$ and ${\bf p}_3$ is shorter than that of $\tilde{\gamma}_2$, because it is a straight line. This resolves case 2.

Case 1: Suppose that the circular arcs of $\gamma_1$ and $\gamma_2$ departing from the point ${\bf p}_1$ belong to different circles. Construct a line $l_1$ between these two circles. On the other hand, construct a line $l_2$ through the point ${\bf p}_3$ which runs perpendicular to the straight segment of $\gamma_1$. Reflect the first part of $\gamma_2$ through the line $l_1$. Observe that the curves $\gamma_1$ and $\tilde{\gamma}_2$ share the circular arc ${\bf p}_1 {\bf p}_2$. Moreover, the straight segment of $\gamma_1$ is shorter than the part of $\gamma_2$ between the point ${\bf p}_2$ and the line $l_2$. Finally, we need to show that the circular arc ${\bf p}_3 {\bf p}_4$ is shorter than the part of $\tilde{\gamma}_2$ between the line $l_2$ and the point ${\bf p}_4$. In particular, we need to show that the circular arc ${\bf p}_3 {\bf q}_2$ is shorter than the straight segment ${\bf q}_1 {\bf q}_2$, where ${\bf q}_1$ is the point of intersection of the line $l_2$ and the curve $\tilde{\gamma}_2$, and ${\bf q}_2$ is the point where the curve $\tilde{\gamma}_2$ meets the arc ${\bf p}_3 {\bf p}_4$. To see this, consider the map $\psi$ which projects the straight segment ${\bf q}_1 {\bf q}_2$ onto the circular arc ${\bf p}_3 {\bf q}_2$. The slope of the circular arc ${\bf p}_3 {\bf q}_2$ through a point ${\bf x}$ is less than or equal to the slope of the straight segment ${\bf q}_1 {\bf q}_2$ through the point $\psi^{-1}({\bf x})$. In other words, the Jacobian determinant of the map $\psi$ is less than or equal to 1, meaning that it shrinks the straight segment ${\bf q}_1 {\bf q}_2$.
\end{proof}

\begin{figure}[h]
    \centering
    \includegraphics[width=\textwidth]{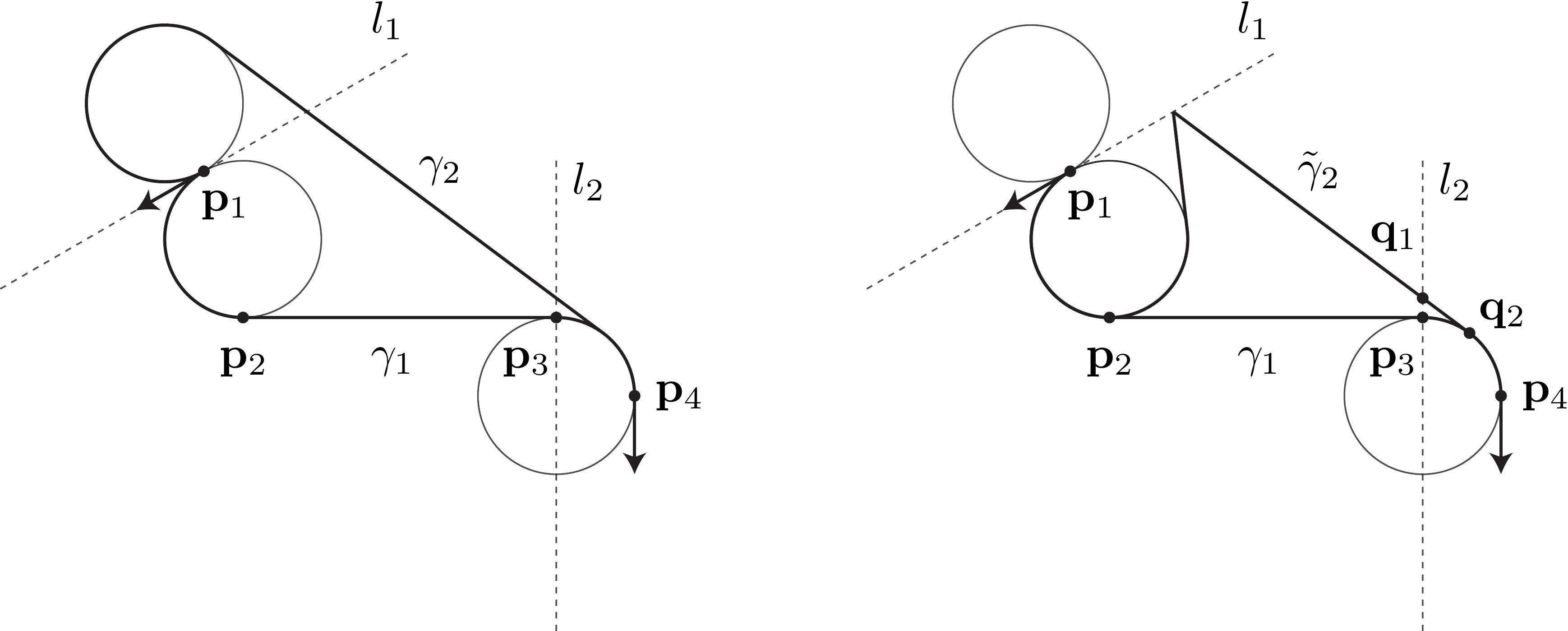}
    \caption{In case 1, the curves $\gamma_1$ and $\gamma_2$ each contain a circular arc belonging to the same circle and, in particular, the arc of $\gamma_1$ is longer than that of $\gamma_2$.}
    \label{fig:3}
\end{figure}

\begin{figure}[h]
    \centering
    \includegraphics[width=\textwidth]{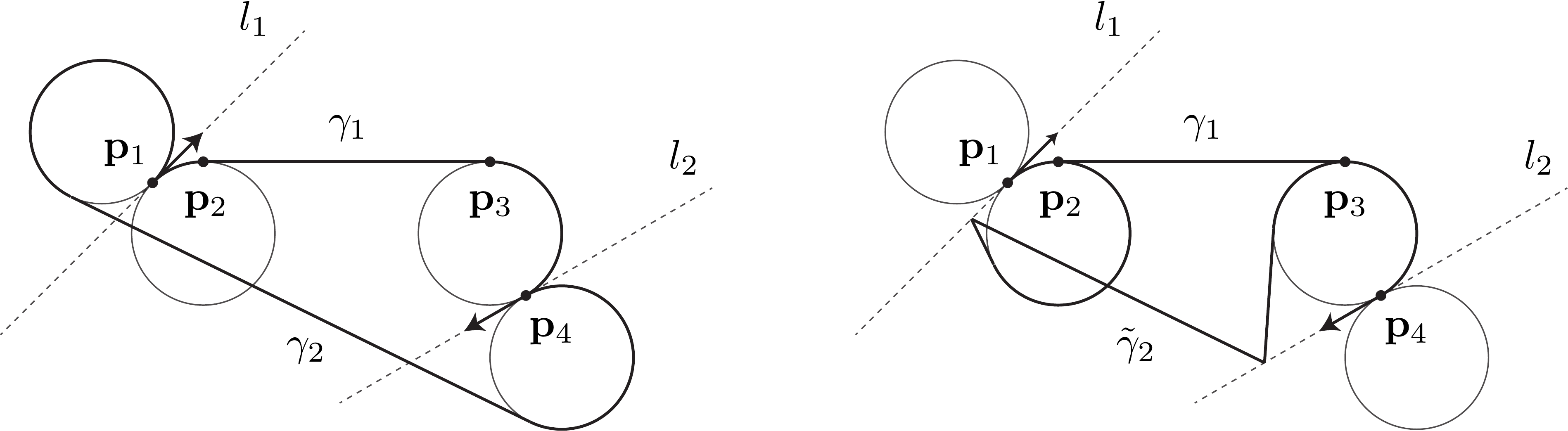}
    \caption{Case 2 covers all possibilities aside from case 1.}
    \label{fig:4}
\end{figure}

The following corollary will be used in Section 6 to obtain an upper bound on the difference in length between an optimal DTSP solution and the optimal ETSP solution for a given point set.

\begin{corollary}
\label{cor:1}
If the circular arc of a $CS$ curve connecting a directed and undirected point is strictly shorter than $\pi$, then it is the shortest amongst both such $CS$ curves.
\end{corollary}

\begin{proof}
The proof requires a simpler version of the argument given for Proposition \ref{pro:1}. Denote the $CS$ curve ${\bf p}_1 {\bf p}_2 {\bf p}_3$ by $\gamma_1$ and suppose that the arc ${\bf p}_1 {\bf p}_2$ is strictly shorter than $\pi$. Let $\gamma_2$ denote the distinct $CS$ curve connecting the directed point $({\bf p}_1, \theta_1)$ and the undirected point ${\bf p}_3$. Construct a line $l$ between the two circles containing the arcs of $\gamma_1$ and $\gamma_2$, respectively. See Figure \ref{fig:5}. Reflect the first part of the curve $\gamma_2$ through this line to obtain a curve $\tilde{\gamma}_2$ whose length is equal to that of $\gamma_2$. The curves $\gamma_1$ and $\tilde{\gamma}_2$ share the circular arc ${\bf p}_1 {\bf p}_2$. However, the part of $\gamma_1$ between the points ${\bf p}_2$ and ${\bf p}_3$ is shorter than that of $\tilde{\gamma}_2$, which completes the proof.
\end{proof}

\begin{figure}[h]
    \centering
    \includegraphics[width=\textwidth]{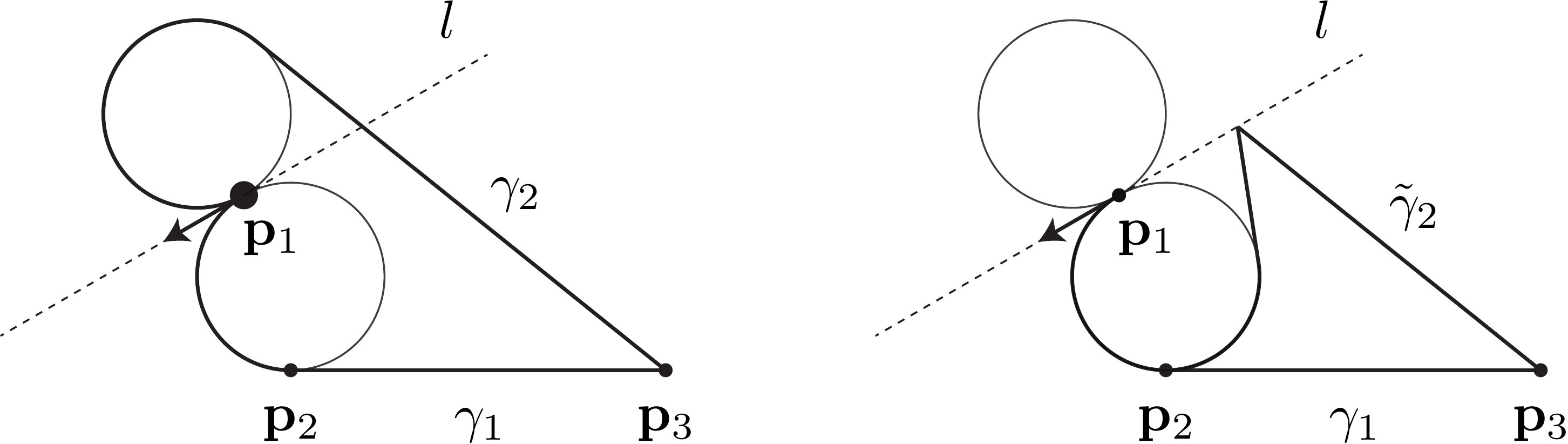}
    \caption{A similar reflection argument applies to the case of a $CS$ curve.}
    \label{fig:5}
\end{figure}

\begin{theorem}
\label{the:3}
The length of a tour through a fixed sequence of points is locally strictly convex as a function of the angles $\theta_1, \theta_2, \ldots, \theta_n$ if each arc of a $CSC$ curve is strictly shorter than $\pi$.
\end{theorem}

\begin{proof}
We first note that by Proposition \ref{pro:1}, there exists a neighbourhood about the given configuration in which the types of the respective Dubins curves which constitute the tour remain fixed. Now, consider the Dubins curve which connects the $i$th and $(i + 1)$th directed points. By Theorem \ref{the:1}, if the configuration is perturbed in the direction of an arbitrary vector $(\omega_1, \omega_2, \ldots, \omega_n)$, then the first variation of the length of the $i$th Dubins curve is given by

\begin{equation*}
(-\omega_i {\bf r}_{12} \times {\bf \hat u}_1 + \omega_2 {\bf r}_{21} \times {\bf \hat u}_1) \cdot {\bf \hat z}.
\end{equation*}

It follows that the second variation of the length of the $i$th Dubins curve is given by

\begin{equation*}
(-\omega_i ({\bf r}_{12}' \times {\bf \hat u}_1 + {\bf r}_{12} \times {\bf \hat u}_1') + \omega_2 ({\bf r}_{21}' \times {\bf \hat u}_1 + {\bf r}_{21} \times {\bf \hat u}_1')) \cdot {\bf \hat z}.
\end{equation*}

We claim that $({\bf r}_{12}' \times {\bf \hat u}_1) \cdot {\bf \hat z} = 0$ and $({\bf r}_{21}' \times {\bf \hat u}_1) \cdot {\bf \hat z} = 0$, whereas $({\bf r}_{12} \times {\bf \hat u}_1') \cdot {\bf \hat z} < 0$ and $({\bf r}_{21} \times {\bf \hat u}_1') \cdot {\bf \hat z} > 0$ when the arcs of the $CSC$ curve are strictly shorter than $\pi$. This proves that the length function is strictly convex, since the perturbation vector was assumed to be arbitrary. It is easy to prove the claim if we view the situation after performing a change of basis. Denote the $CSC$ curve by ${\bf p}_1 {\bf p}_2 {\bf p}_3 {\bf p}_4$. Referring to Figure \ref{fig:6}, suppose that instead of rotating the first disk about the point ${\bf p}_1$, we instead fix the disk and rotate the rest of the figure about it. In this case, it is clear that the point of tangency ${\bf r}_{12}$ travels around the perimeter of the first disk, in which case ${\bf r}_{12}' \parallel {\bf \hat u}_1$ and $({\bf r}_{12}' \times {\bf \hat u}_1) \cdot {\bf \hat z} = 0$. On the other hand, the vector ${\bf \hat u}_1'$ points orthogonal to ${\bf \hat u}_1$ and into the first disk, so $({\bf r}_{12} \times {\bf \hat u}_1') \cdot {\bf \hat z} < 0$ when the first arc is strictly shorter than $\pi$. An analogous argument demonstrates that $({\bf r}_{21}' \times {\bf \hat u}_1) \cdot {\bf \hat z} = 0$ and $({\bf r}_{21} \times {\bf \hat u}_1') \cdot {\bf \hat z} > 0$.
\end{proof}

\begin{figure}[h]
    \centering
    \includegraphics[width=0.7\textwidth]{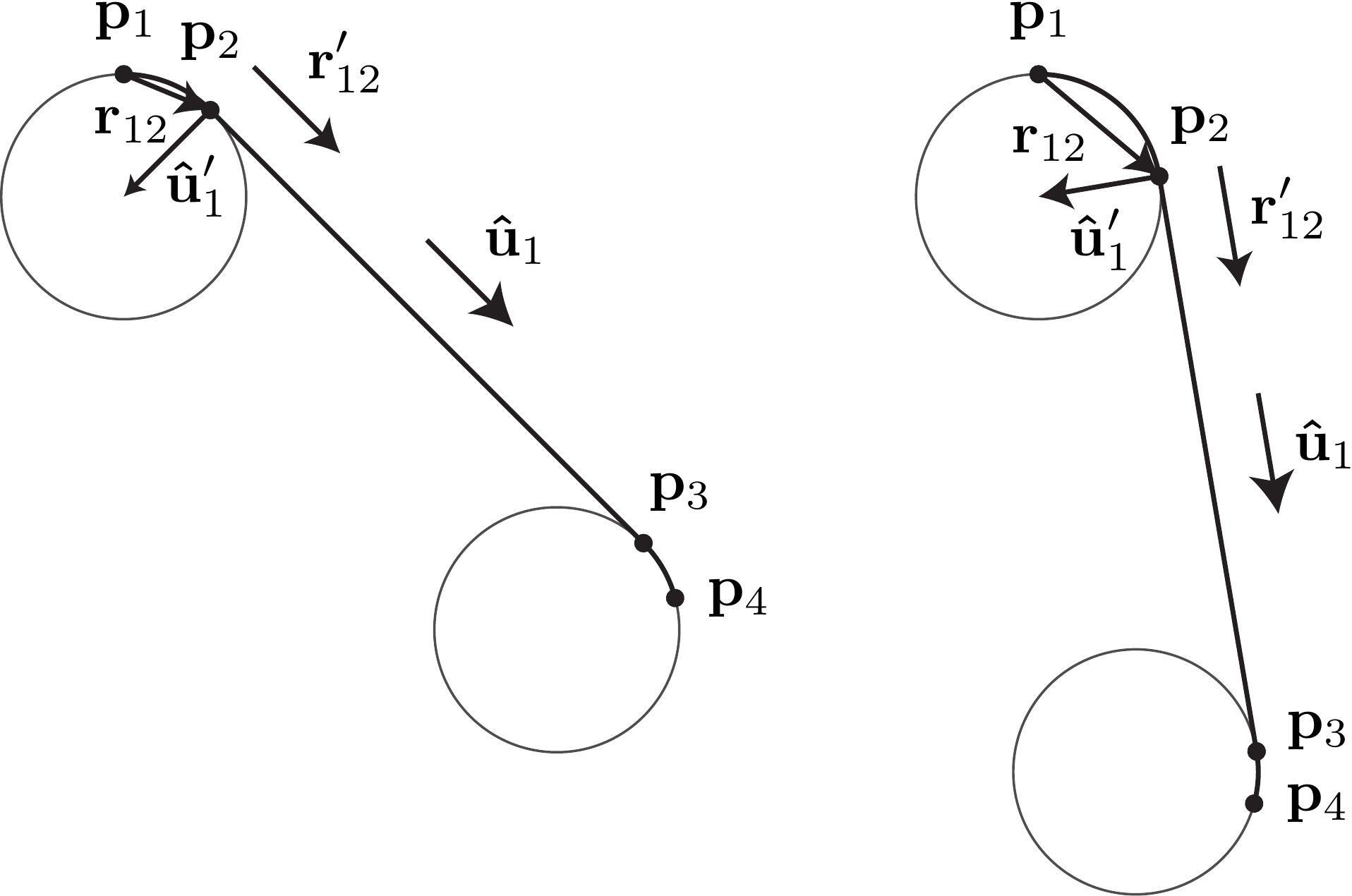}
    \caption{Perform a change of basis, fixing the first disk and rotating the rest of the figure about it.}
    \label{fig:6}
\end{figure}

\begin{theorem}
\label{the:4}
If each arc of a $CSC$ curve of the tour is strictly shorter than $\pi$, then the optimality gap between the incumbent objective value $L(\boldsymbol \theta)$ and the locally optimal objective value $L(\boldsymbol \theta^*)$ of the gradient descent method is bounded from above by $2 \sqrt{n} \pi \| \nabla L(\boldsymbol \theta) \|$.
\end{theorem}

\begin{proof}
By local convexity,

\begin{equation*}
L(\boldsymbol \theta^*) \geq L(\boldsymbol \theta) + (\boldsymbol \theta^* - \boldsymbol \theta) \nabla L(\boldsymbol \theta).
\end{equation*}

It follows that

\begin{align*}
L(\boldsymbol \theta) - L(\boldsymbol \theta^*) &\leq (\boldsymbol \theta - \boldsymbol \theta^*) \cdot \nabla L(\boldsymbol \theta)\\
&\leq \max_{\Delta \boldsymbol \theta} \Delta \boldsymbol \theta \cdot \nabla L(\boldsymbol \theta),
\end{align*}

where $\Delta \boldsymbol \theta$ is an arbitrary step within the region of the domain where the respective types of the $CSC$ curves remain fixed. Since the step $\Delta \boldsymbol \theta$ lies in the domain $[-\pi, \pi)^n$, its magnitude is no larger than $2 \sqrt{n} \pi$. It follows that the righthand side is maximised when $\Delta \boldsymbol \theta = 2 \sqrt{n} \pi \nabla L(\boldsymbol \theta)/\|\nabla L(\boldsymbol \theta)\|$. That is,

\begin{equation*}
L(\boldsymbol \theta) - L(\boldsymbol \theta^*) \leq 2 \sqrt{n} \pi \| \nabla L(\theta) \|.
\end{equation*}
\end{proof}

\begin{figure}[h]
    \centering
    \includegraphics[width=0.45\textwidth]{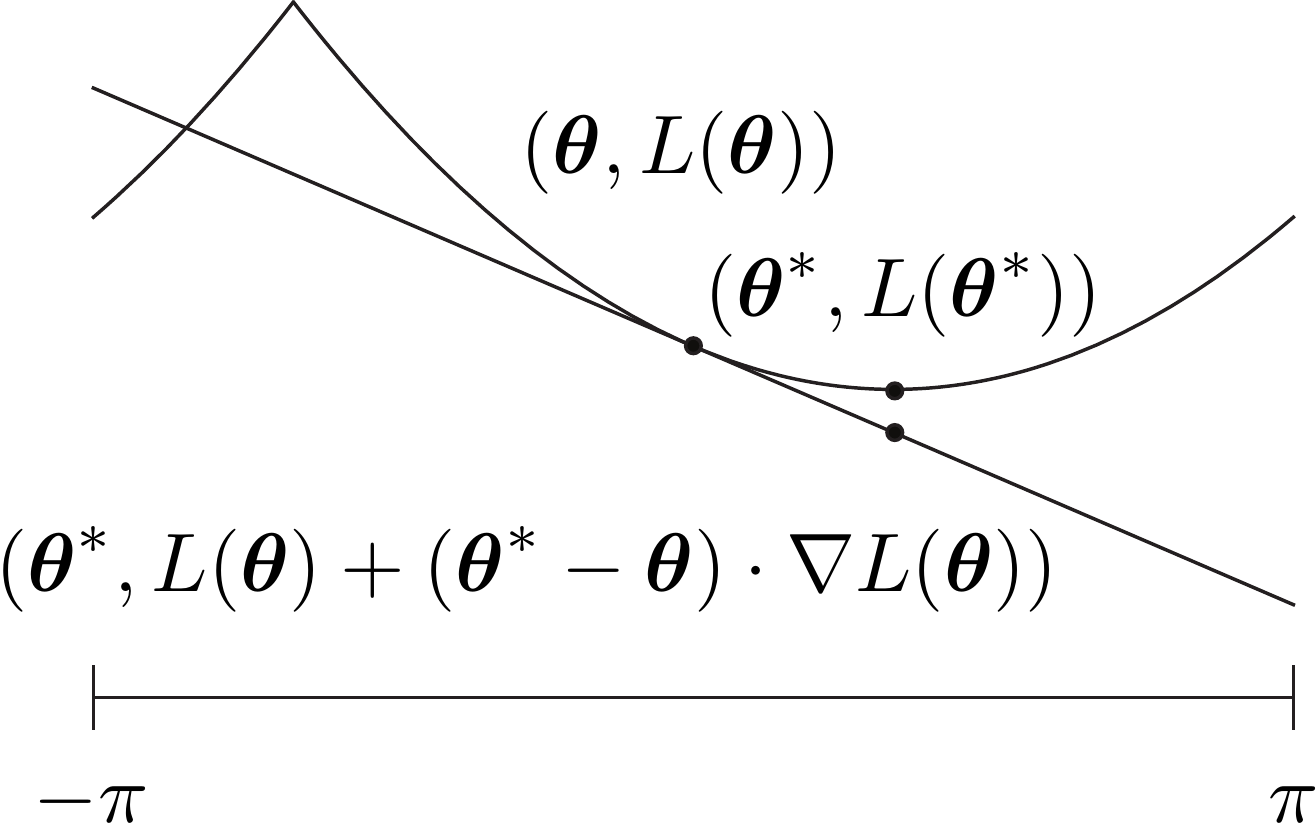}
    \caption{The graph of the length function is locally bounded from below by a supporting hyperplane.}
    \label{fig:7}
\end{figure}

\section{The gradient descent method}
\label{sec:5}
The pseudocode for the gradient descent method for minimising the length of a curvature-constrained tour through a given sequence of points is below.

\begin{algorithm}[H]
\SetAlgoLined
\KwResult{An approximation to the optimal configuration of directions ${\boldsymbol \theta} = (\theta_1, \theta_2, \ldots, \theta_n)$ for a given sequence of points}
 initialization: For $i = 1, 2, \ldots, n$, assign ${\bf v}_i \gets {\bf \hat e}_{\mbox{mod}(i - 1, n, 1)} + {\bf \hat e}_i$ and $(\cos \theta_i, \sin \theta_i) \gets {\bf \hat v}_i$\;
 Fix the learning rate $\alpha$ and tolerance $\epsilon$\;
 \While{$\| \nabla L(\theta) \| > \epsilon$}{
  ${\boldsymbol \theta} = {\boldsymbol \theta} + \alpha {\boldsymbol \tau}$\;
 }
 \caption{The gradient descent method}
\end{algorithm}

\begin{remark}
At each iteration, in order to compute the torques and, by implication, the gradient, we compute the shortest Dubins curves between the various pairs of consecutive points. Since the Dubins curves can be computed independently of one another, as is the case for the torques, the computation time grows linearly in the number of points. These computations may also be performed in parallel.
\end{remark}

\section{The bounding procedure}
\label{sec:6}
We will now describe the bounding procedure to be used in conjunction with the gradient descent method for solving the Dubins traveling salesman problem. First, we solve the euclidean traveling salesman problem. In particular, we solve the ETSP using the Dantzig–Fulkerson–Johnson formulation \cite{Dantzig}. Let $L_E$ denote the length of the optimal ETSP solution. The corresponding DTSP solution produced by the gradient descent method immediately gives an upper bound of $L_D$ on the length of a curvature-constrained tour through the given set of nodes. However, the order of the optimal ETSP solution may not match that of the optimal DTSP solution. So we attempt to generate a distinct solution to the ETSP whose length is between $L_E + \epsilon$ and $L_D$ for some small tolerance $\epsilon$. The lower and upper bounds $L_E + \epsilon$ and $L_D$ may be introduced to the Dantzig–Fulkerson–Johnson integer linear program as linear inequalities. If we succeed in generating a distinct solution, then we may update the values of $L_E$ and $L_D$ before attempting to generate another distinct solution. To summarise, we provide the pseudocode below.

\begin{algorithm}[H]
\SetAlgoLined
\KwResult{A solution to the Dubins traveling salesman problem}
 initialization $i = 1$, $L_E \gets 0$ and $L_D \gets \infty$\;
 \While{While True}{
  Solve the ETSP with bounds $L_E + \epsilon \leq L \leq L_D$\;
  \eIf{A solution $\Gamma_i$ exists}{
   Update the values of $L_E$ and $L_D$\;
   i = i + 1\;
   }{
   Break\;
  }
 }
 Select the shortest DTSP solution.
 \caption{The bounding procedure}
\end{algorithm}

\begin{remark}
Note that a tour may self-intersect locally when two adjacent edges intersect or non-locally when two nonadjacent edges intersect. The presence or absence of a local self-intersection gives rise to two topologically distinct tours and hence to two local minima. However, the bounding procedure does not distinguish between two tours which differ by a local self-intersection, because the two tours are associated with the same sequence of points.

The issue of local self-intersections centres around the notion of a ``sharp turn'' \cite{Goaoc}. Consider three consecutive points ${\bf x}_1, {\bf x}_2$ and ${\bf x}_3$, and denote the angle between the straight segments ${\bf x}_1 {\bf x}_2$ and ${\bf x}_2 {\bf x}_3$ by $\phi$. The triplet ${\bf x}_1, {\bf x}_2, {\bf x}_3$ is called a sharp turn if $\phi \leq \frac{\pi}{2}$ and one of the points ${\bf x}_1$ and ${\bf x}_2$ is within distance $4\rho$ of the straight segment formed by the other two points. If the triplet ${\bf x}_1, {\bf x}_2, {\bf x}_3$ does not form a sharp turn, then no globally shortest tour has a local self-intersection involving these three points. Without going into details, there is a sort of exchange argument to prove this point, which is indicated in Figure \ref{fig:8}. If a triplet does not form a sharp turn, then a self-intersection involving the three points can be eliminated and the tour shortened.

\begin{figure}[h]
    \centering
    \includegraphics[width=0.5\textwidth]{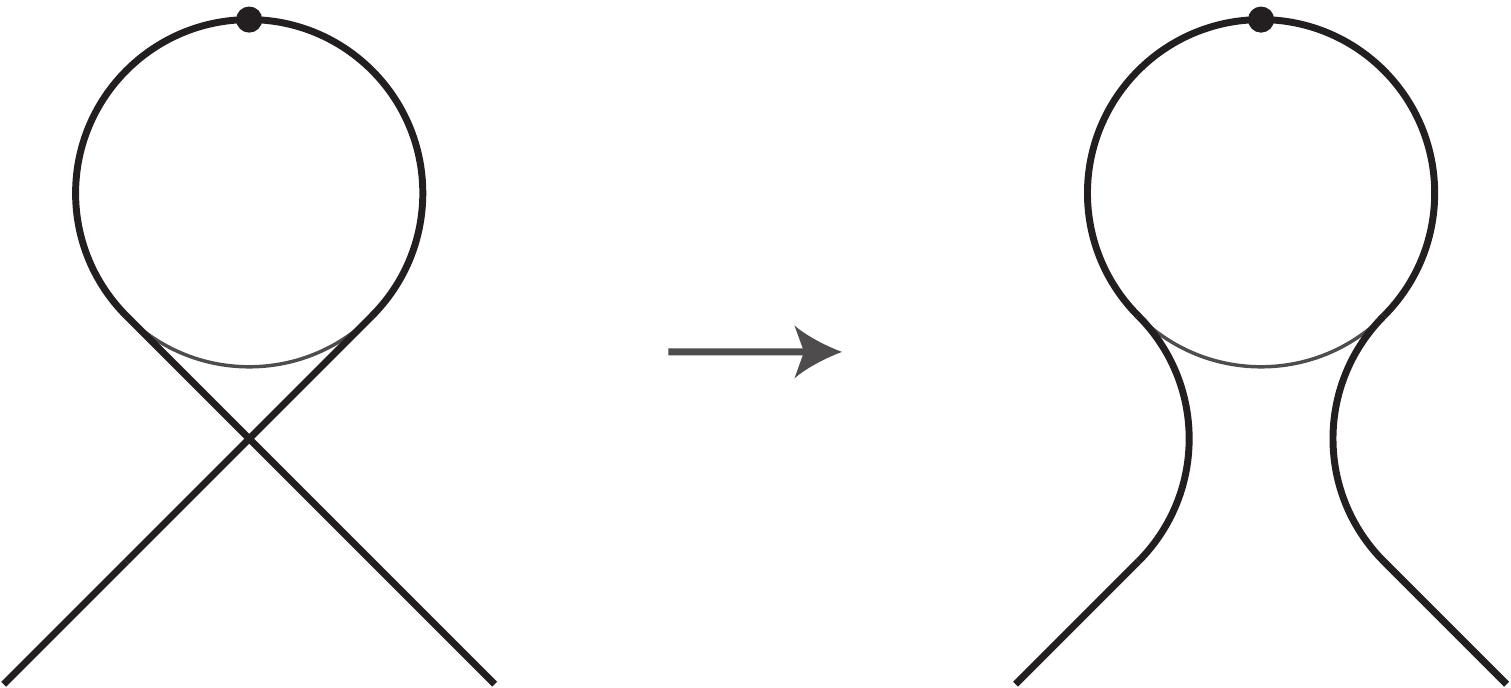}
    \caption{An example of the elimination of a self-intersection using an exchange argument.}
    \label{fig:8}
\end{figure}

We augment the bounding procedure as follows in order to guarantee that it finds the globally shortest tour. Wherever a sharp turn is identified in a sequence of points, we consider the possibility of flipping the direction of the unit tangent vector $(\cos \theta, \sin \theta)$ associated with the second of the three points, that is, $\theta \mapsto \theta + \pi$. All combinations of such flips are considered.
\end{remark}

Now, let us refer to the ratio of the minimum distance between any two points and the minimum turning radius as the scale $s$ of the Dubins traveling salesman problem. We can obtain a crude bound on the gap between the length of the optimal ETSP solution and any other ETSP solution which needs to be considered before we can be sure that we have found the optimal DTSP solution. The length of any ETSP solution is at least $ns$, assuming the minimum turning radius is normalised to unity. According to the following proposition, we only have to consider solutions of the ETSP which have length at most $1 + \pi/s$ times the length of the optimal ETSP solution to ensure that we find the optimal DTSP solution.

\begin{proposition}
In the long path case, that is, the case where any two points are separated by at least four times the minimum turning radius $\rho$, for a point set of size $n$, the length of an optimal $DTSP$ exceeds the length of the optimal $ETSP$ solution by no more than $\pi n$.
\end{proposition}

\begin{proof}
The proof uses the alternating algorithm for constructing a suboptimal DTSP solution, a more elaborate explanation of which is given along with illustrations in the work of Savla et al \cite{Savla}. We consider two cases, namely, the case where $n$ is even and the case where $n$ is odd. If $n$ is even, then retain every second edge of the ETSP solution, while replacing every other edge by a shortest Dubins curve of type $CSC$ in such a way that the tour contains no cusp points. By Proposition \ref{pro:1}, each of the $n$ arcs of the tour is no longer than $\pi$. Given a shortest Dubins curve ${\bf p}_1 {\bf p}_2 {\bf p}_3 {\bf p}_4$ of type $CSC$, each of whose arcs are no longer than $\pi$, the corresponding edge ${\bf p}_1 {\bf p}_4$ of the ETSP solution is at least as long as the straight segment ${\bf p}_2 {\bf p}_3$, because its projection onto the line through ${\bf p}_2$ and ${\bf p}_3$ contains the straight segment ${\bf p}_2 {\bf p}_3$. The conclusion for the case where $n$ is even follows.

On the other hand, if $n$ is odd, then retain every second edge of the ETSP solution. The subset of edges from the ETSP solution which we retain determine the directions through all but the 1st point. In order to determine the direction through the 1st point, we construct a shortest Dubins curve of type $CS$ from the $n$th to the 1st point. Once again, the corresponding edge of the ETSP solution is at least as long as the straight segment of the $CS$ curve, whose arc is no longer than $\pi$, according to Corollary \ref{cor:1}. The remaining edges of the tour are shortest Dubins curves of type $CSC$, from which the conclusion follows.
\end{proof}

\section{Experimental results}
\label{sec:7}
In this section, we examine the results of applying the bounding procedure to a large number of approximately random problem instances. To explain the setup, we start by generating 1000 approximately random problem instances, where each problem instance comprises nine nodes in a $12 \times 12$ box. The nodes are separated from one another by at least four times the minimum turning radius, so that each problem instance is an instance of the long path case. Each problem instance is generated via the Metropolis-Hastings algorithm. For more details, see Chapter 5 in the textbook by Liu \cite{Liu}. The relative proximity of the nodes means that each problem instance is far removed from the limiting case where the minimum turning radius is negligible in comparison with the minimum distance between any two points. Accordingly, we expect the solution to the DTSP to differ significantly from the solution to the ETSP. In other words, we expect the bounding algorithm to explore a number of ETSP solutions before terminating. On the other hand, as we scale the problem by increasing both the length of the box and the minimum distance between any two points, we expect the bounding algorithm to explore fewer and fewer ETSP solutions before terminating. Figure \ref{fig:9} (left) shows the average number of ETSP solutions with error bars as the scale of the problem is varied from 4 through 40. For a given sequence of points, denote the length of the incumbent solution and that of the optimal solution by $L$ and $L^*$, respectively. We define the percentage error in the length of the incumbent solution as $100 \frac{L - L^*}{L^*}$. Applying the bound in Theorem \ref{the:4}, we run gradient descent til the percentage error in the length of the incumbent solution is guaranteed to be no more than $10^{-1}$. The bound also informs us about the percentage error in the length of the seed solution. Figure \ref{fig:9} (right) shows the average percentage error in the length of the seed solution with error bars as the scale of the problem is varied from 4 through 40. The seed solution is of very high quality, meaning that it may reliably be used as a heuristic for solving the DTSP, and it may not be necessary from a practical standpoint to use the gradient descent method.

\begin{figure}[h]
    \centering
    \includegraphics[width=\textwidth]{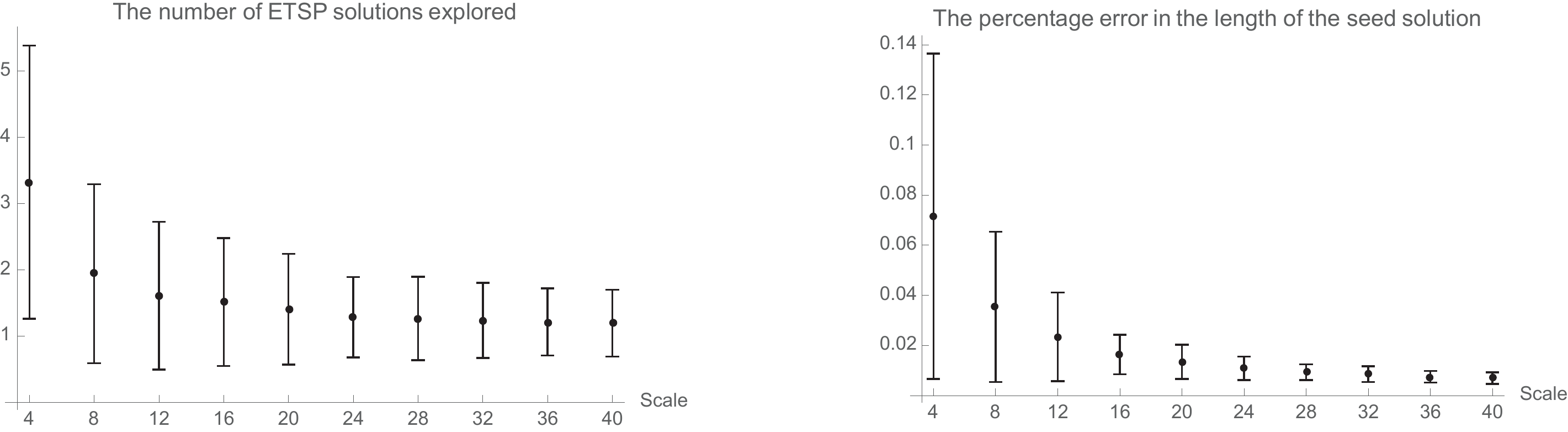}
    \caption{}
    \label{fig:9}
\end{figure}

\section{A generalisation to variable points}
\label{sec:8}
In the final section, we explain how to generalise the gradient descent method to the case where not only the directions of the tour through the points but also the positions of the points themselves may be varied. For example, in forthcoming work, the points are allowed to slide along straight lines. Alternately, in the work of Chen et al., each point is confined to its own disk \cite{Chen}. The generalisation requires the following proposition with adapted notation from the work of Ayala et al. \cite[Proposition 5.2]{Ayala}.

\begin{proposition}
Let ${\bf p}_1 {\bf p}_2 {\bf p}_3 {\bf p}_4$ denote a Dubins curve of type $CSC$, and let ${\bf \hat u}$ denote the normalisation of the vector ${\bf p}_3 {\bf p}_2$. If the point ${\bf p}_4$ is varied with velocity ${\bf v}$, then the first variation of the length of the curve is given by

\begin{equation*}
-{\bf v} \cdot {\bf \hat u}.
\end{equation*}
\end{proposition}

\begin{proof}
Referring to Figure \ref{fig:2} (right), denote the midpoint of the straight segment ${\bf p}_2 {\bf p}_3$ by ${\bf m}$. Since the point ${\bf m}$ lies midway between the two circles around which the curve wraps, it must travel with velocity $\frac{1}{2}{\bf v}$. By Lemma \ref{lem:1}, the first variation of the curve ${\bf p}_1 {\bf p}_2 m$ is

\begin{equation*}
-\frac{1}{2}{\bf v} \cdot {\bf \hat u}.
\end{equation*}

On the other hand, subtracting ${\bf v}$ from all elements in Figure \ref{fig:2} (right), we see that the first variation of the length of the curve ${\bf m} {\bf p}_3 {\bf p}_4$ is

\begin{equation*}
-(-\frac{1}{2}{\bf v}) \cdot (- {\bf \hat u}) = -\frac{1}{2}{\bf v} \cdot {\bf \hat u}.
\end{equation*}

Summing these two terms, we obtain for the first variation of the length of the curve ${\bf p}_1 {\bf p}_2 {\bf p}_3 {\bf p}_4$ the expression

\begin{equation*}
-{\bf v} \cdot {\bf \hat u}.
\end{equation*}
\end{proof}

\begin{remark}
As for Lemma \ref{lem:1}, the unit vector ${\bf \hat u}$ may be interpreted as the force exerted upon the point ${\bf p}_4$ by the elastic representing the Dubins curve.
\end{remark}

Recalling that the unit vector which points along the $i$th straight segment and has the same orientation as the tour is denoted by ${\bf \hat u}_i$, the partial derivative of the length of the tour with respect to the position of the $i$th point is given by

\begin{equation*}
\frac{\partial L}{\partial {\bf x}_i} = {\bf \hat u}_{\mbox{mod}(i - 1, n, 1)} - {\bf \hat u}_i.
\end{equation*}

These partial derivatives must be incorporated into the expression for the gradient if the gradient descent method is to be applied to the case where the positions of the points may be varied. If the motions of the points are constrained, e.g., along straight lines, then the partial derivative above must be projected onto the curve which constrains the motion of the $i$th point.

%



%

%

\end{document}